\documentclass[12pt]{amsart}

\newtheorem{theorem}{Theorem}
\newtheorem{lemma}[theorem]{Lemma}
\newtheorem{proposition}[theorem]{Proposition}
\newtheorem{corollary}[theorem]{Corollary}
\newtheorem{remark}[theorem]{Remark}

\DeclareMathOperator{\aut}{Aut}

\DeclareMathOperator{\id}{id}
\newcommand{\NN}{\mathbb N}
\newcommand{\RR}{\mathbb R}
\newcommand{\UU}{\mathbb U}
\newcommand{\HH}{\mathbb H}

\newcommand{\CC}{\mathbb C}

\newcommand{\OO}{\mathcal O}
\newcommand{\EE}{\mathcal E}

\begin{document}

\title{A characterization of holomorphic mappings on a poly-plane}
\author{Armen Edigarian}
\address{Wydzia\l\ Matematyki i Informatyki, Uniwersytet Jagiello\'nski, \L o\-ja\-sie\-wi\-cza 6, 30-348 Krak\'ow, Poland}
\email{armen.edigarian@uj.edu.pl}
\thanks{The author was supported in part by the Polish National Science Centre (NCN) grant no. 2015/17/B/ST1/00996.}

\begin{abstract} 
We show that any function $f:\HH^n\to\HH$ with $f(z+c)=f(z)+c$, $z\in\HH^n$, for some $c>0$ has a property that any limit function of a family
$\{\frac{f(tz)}{t}\}_{t>0}$ when $t\to\infty$ is linear.
\end{abstract}

\maketitle

\section{Introduction}
Our paper is motivated by results of V. Markovic and D. Gekhtman \cite{Markovic, Gekhtman, Gekhtman-Markovic}.
Recently, Vladimir Markovic \cite{Markovic} solved a long-standing open problem in Teichm\"uller theory, by showing that
Carath\'eodory and Kobayashi metrics on Teichmuller spaces are generally different. One of the crucial step in the proof is the following rigidity result.
Here, $\HH$ denote the upper half-plane.
\begin{theorem}\label{thm:1a}
Let $f:\HH^n\to\HH$ be a holomorphic function satisfying the conditions:
\begin{enumerate}
\item $f(\lambda,\dots,\lambda)=\lambda$,
\item $\frac{\partial f}{\partial z_j}(\lambda,\dots,\lambda)=\alpha_j$, $j=1,\dots,n$, where $\alpha_j\in[0,1]$;
\item $\sum_{j=1}^n\alpha_j=1$;
\end{enumerate}
for any $\lambda\in\HH$. Assume that there exists a constant $c>0$ such that
\begin{equation}\label{eq:1a}
f(z_1+c,\dots,z_n+c)=f(z_1,\dots,z_n)+c\quad\text{ for any } z_1,\dots,z_n\in\HH.
\end{equation}
Then $f(z)=\sum_{j=1}^n \alpha_j z_j$.
\end{theorem}
Theorem~\ref{thm:1a} for n=2 is proved in \cite{Markovic} and for any $n\ge2$ in \cite{Gekhtman}. If the proof for $n=2$ in \cite{Markovic}
is complex analytic, the proof for $n\ge2$ is based on some results from ergodic theory. 
The aim of our paper is to analyze the properties of holomorphic mappings $f:\HH^n\to\HH$ satisfying \eqref{eq:1a} for some constant $c>0$. 
In particular, using the ideas from \cite{Markovic}, we show that any limit holomorphic function of the family $\{\frac{1}{t}f(tz)\}_{t>0}$ is linear, i.e., any holomorphic function $h:\HH^n\to\HH$ such that
$\lim_{j\to\infty}\frac{f(t_jz)}{t_j}=h(z)$ for a sequence $t_j\to\infty$ is linear. In consequence, we have the following generalization of Theorem~\ref{thm:1a} for any $n\ge2$ by purely analytic proof.
\begin{theorem}\label{thm:1} Let $f:\HH^n\to\HH$ be a holomorphic mapping satisfying \eqref{eq:1a} for some $c>0$.
Assume that there exists a sequence $t_j\to\infty$ and a holomorphic
mapping $h:\HH^n\to\HH$ such that $\frac1{t_j}f(t_jz)\to h(z)$ for any $z\in\HH^n$.
Then $h$ is linear. 

Moreover, if there exists a point $w\in\HH^n$ such that $\Im(f(w)-h(w))=0$ then there is a constant $\alpha\in\RR$ with $f=h+\alpha$.
\end{theorem}

\section{Proofs}
Let us start with the following simple remark (see discussion on page 19 in  \cite{Markovic}).
\begin{proposition}
Let $f:\HH^n\to\HH$ be a holomorphic function and let $c>0$ be a constant so that
$$
f(z+c)=f(z)+c\quad\text{ for any }z\in\HH^n.
$$
If $\frac{1}{t_j}f(t_jz)\to h(z)$, $z\in\HH^n$, when $j\to\infty$, where $h:\HH^n\to\HH$ is a holomorphic mapping then
\begin{equation}\label{eq:21}
h(z+s)=h(z)+s\quad\text{ for any } z\in\HH^n\text{ and any }s\in\RR.
\end{equation}
\end{proposition}

\begin{proof} Note that  $\frac{1}{t_j}f(t_jz)\to h(z)$ locally uniformly in $\HH^n$.
Fix $s>0$. For any $j\ge 1$ we have $t_j s=\ell_j c+q_j$, where $\ell_j\in\NN$ and $0\le q_j<c$.
Then 
$$
h(z+s)=\lim_{j\to\infty}\frac{1}{t_j}f(t_jz+t_js)=
\lim_{j\to\infty}\frac{1}{t_j}\left(f\Big(t_j(z+\frac{q_j}{t_j})\Big)+\ell_j c\right)=h(z)+s.
$$
\end{proof}

In \cite{Markovic} (see also \cite{Gekhtman}) it is proved.
\begin{theorem}\label{thm:5} Let $h:\HH^n\to\HH$ be a holomorphic mapping such that
$$
h(z+s)=h(z)+s\quad\text{ for any }z\in\HH^n\text{ and any }s\in\RR.
$$
Then there exist $\alpha_0\in\CC$ and $\alpha_1,\dots,\alpha_n\in\RR$ such that $\sum_{k=1}^n\alpha_k=1$ and
$$
h(\lambda_1,\dots,\lambda_n)=\alpha_0+\sum_{j}\alpha_j\lambda_j,
$$
where $\Im\alpha_0\ge0$, $\alpha_j\ge0$ and $\sum_{j=1}^n\alpha_j=1$.
\end{theorem}
Frist recall the following version of a well-known result (see e.g. Proposition B.1 in \cite{Markovic}, Lemma 6.1 in \cite{Gekhtman}). For the sake of completeness, we give a proof.
\begin{lemma}\label{lemma:5}
Let $f:\CC^n\to\CC$ be a holomorphic mapping such that for some $A,B>0$ we have
$$
\Re(f(z))\le A\max\{|z_1|,\dots,|z_n|\}+B,\quad\text{ for any } z\in\CC^n.
$$
Then there exist $\alpha_0,\alpha_1,\dots,\alpha_n\in\CC$ such that $f(z)=\alpha_0+\sum_{j=1}^n \alpha_j z_j$.
\end{lemma}

\begin{proof} Without loss of generality, we may assume that $n=1$.
The proof follows from the ideas from Nevannlina theory (see e.g. \cite{Ru}).
Fix $R>0$. For any $z\in\CC$ with $|z|<R$ we have
$$
\Re(f(z))=\frac{1}{2\pi}\int_{0}^{2\pi} \Re(f(Re^{i\theta}))\Re\left(\frac{Re^{i\theta}+z}{Re^{i\theta}-z}\right)d\theta.
$$
And, therefore,
\begin{align*}
\Re(f'(z))=\frac{1}{2\pi}\int_{0}^{2\pi} \Re(f(Re^{i\theta}))\frac{2R e^{i\theta}}{(Re^{i\theta}-z)^2}d\theta,\\
\Re(f''(z))=\frac{1}{2\pi}\int_{0}^{2\pi} \Re(f(Re^{i\theta}))\frac{-4Re^{i\theta}}{(Re^{i\theta}-z)^3}d\theta.
\end{align*}
Using the equality $|x|=2x_{+}-x$ we get
\begin{multline}
|\Re(f''(z))|\le\frac{1}{2\pi}\cdot\frac{4R}{(R-|z|)^3}\cdot\int_{0}^{2\pi} |\Re(f(Re^{i\theta}))|d\theta=\\
\frac{1}{2\pi}\cdot\frac{4R}{(R-|z|)^3}\left(2\int_{0}^{2\pi} (\Re(f(Re^{i\theta})))_{+}d\theta-2\pi\Re f(0)\right).
\end{multline}
Using the inequality $(\Re(f(Re^{i\theta}))_+\le AR+B$ and taking $R\to\infty$ we get $\Re (f'')=0$.
\end{proof}

\begin{proof}[Proof of Theorem~\ref{thm:5}] Fix $N\in\NN$ and put 
$$
H_N(z_1,\dots,z_{n-1})=f(z_1+\lambda,\dots,z_{n-1}+\lambda,\lambda)-\lambda
$$
where $\Im z_j>-N$ and $\Im\lambda>N$. Note that the right hand side does not depend on $\lambda$ and on  $N$. Hence,
$H=\lim_{N\to\infty} H_N\in\OO(\CC^{n-1})$ and 
$$
f(z_1,\dots,z_n)=z_n+H(z_1-z_n,\dots,z_{n-1}-z_n).
$$
We know that $f:\HH^n\to\HH$, hence 
\begin{equation}\label{eq:6}
\Im (z_n+H(z_1-z_n,\dots,z_{n-1}-z_n))>0
\end{equation}
when $\Im z_1,\dots,\Im z_n>0$.
Let us show that for any $w\in\CC^{n-1}$ we have 
\begin{equation}\label{eq:7}
\Im(-H(w))\le 2\max\{|w_1|,\dots,|w_{n-1}|\}.
\end{equation}
Indeed, in the inequality \eqref{eq:6} take $z_1=\dots=z_n\in\HH$. Then $\Im H(0)\ge0$.
If $w\in\CC^{n-1}$ is such that $w\not=0$ then put
$$
z_n=2i\max\{|w_1|,\dots,|w_{n-1}|\}
$$
and $z_j=z_n+w_j$, $j=1,\dots,n-1$. Note that $z_1,\dots,z_n\in\HH$. From the inequality \eqref{eq:6} we get the inequality \eqref{eq:7}.
From Lemma~\ref{lemma:5} we have $H(w)=a_0+\sum_{j=1}^{n-1}a_jw_j$. And, therefore,
$$
f(\lambda_1,\dots,\lambda_n)=\alpha_0+\sum_{j}\alpha_j\lambda_j,
$$
where $\alpha_j\ge0$ and $\sum_{j=1}^n\alpha_j=1$.
\end{proof}

For a convex domain $\Omega\subset\CC^n$ we denote by $d_{\Omega}$ its hyperbolic distance. Recall that
$$
d_{\HH}(a,b)=\frac12\log\frac{|a-\bar b|+|a-b|}{|a-\bar b|-|a-b|}=\log
\frac{|a-\bar b|+|a-b|}{2\sqrt{\Im a\Im b}},
$$
where $a,b\in\HH$. We have also 
$$
d_{\HH^n}(z,w)=\max_{j=1,\dots,n} d_{\HH}(z_j,w_j),
$$
where $z=(z_1,\dots,z_n), w=(w_1,\dots,w_n)\in\HH^n$.

As a simple corollary of the above formula we get.
\begin{lemma} Let $a,b\in\HH$ be such that $|a|<|b|$. Then
$$
\left| d_{\HH}(a,b)-\frac12\log\frac{|b|^2}{\Im a\Im b}\right|\le \frac{|a|}{|b|}.
$$
\end{lemma}

From this result it is easy follows.
\begin{corollary} Let $f:\HH^n\to\HH$ be a holomorphic mapping and let $w=(iw_1,\dots,iw_n)\in\HH^n$, $w_1,\dots,w_n>0$, be a fixed point. 
Assume that for a sequence $t_j\to\infty$ we have
$\lim_{j\to\infty}\frac{f(t_j w)}{t_j}=a$, where $a\in\HH$. 
Then for any $z\in\HH^n$ such that $\Im z=\Im w$ we have
$$
\Im f(z)\ge \frac{|a|^2}{\Im a}
\ge\Im a.
$$
\end{corollary}

\begin{proof} Fix a point $z\in\HH^n$ such that  $\Im z=\Im w$. Then
$$
d_{\HH}(f(z),t_ja)=d_{\HH}(f(z),f(t_j w))+d_{\HH}(f(t_j w), t_j a).
$$
Note that
$$
d_{\HH}(f(t_j w), t_j a)=d_{\HH}\Big(\frac{f(t_j w)}{t_j}, a\Big)\to0\quad\text{ when }j\to\infty.
$$
We have
$$
d_{\HH}(f(z),f(t_j w))\le d_{\HH^n}(z,t_j w)=\max_{k} d_{\HH}(z_k,t_j w_k)=\frac12\log t_j+o(1)
$$
and
$$
d_{\HH}(f(z),t_ja)=\frac12\log t_j+\frac12\log\frac{|a|^2}{\Im f(z)\Im a}+o(1).
$$
Taking $j\to\infty$ we get $\Im f(z)\ge \frac{|a|^2}{\Im a}\ge\Im a$.
\end{proof}

\begin{proof}[Proof of Theorem~\ref{thm:1}] Note that for a linear function we have $\Im h(z)=\Im h(i\Im z)$. 
From the above result we have
$$
\Im f(z)\ge\Im h(z),\quad\text{ for any } z\in\HH^n.
$$
Moreover, if $\Im(f(z_0)-h(z_0))\le0$ then a function $F(z)=e^{i(f(z)-h(z))}$ is such that $|F|\le 1$ and $|F(z_0)|=1$. So, from the maximum principle we get $F=e^{is}$ where $s\in\RR$, and, therefore, $f=h+c$, where $c\in\RR$.
\end{proof}

\begin{remark}
Let $f:\HH^n\to\HH$ be a holomorphic function and let $m>0$ be such that $f(mz)=m f(z)$ for any $z\in\HH^n$.
It is immediate, if $f_1, f_2: \HH^n\to\HH$ are holomorphic mappings such that $f_j(mz)=mf_j(z)$ then
$a_1f_1+a_2f_2:\HH^n\to\HH$, $a_1,a_2>0$, also has the same property.

Note that for any $\alpha_1,\dots,\alpha_n\in[0,1]$, $\sum_{j=1}^n\alpha_j=1$, the function
$$
g_{\alpha_1,\dots,\alpha_n}(z)=z_1^{\alpha_1}\dots z_n^{\alpha_n}:\HH^n\to\HH
$$
has the property $g_{\alpha_1,\dots,\alpha_n}(m z)=mg_{\alpha_1,\dots,\alpha_n}(z)$. So, for $n\ge2$ there exist a big family of functions which satisfy the property $f(mz)=m f(z)$.
 \end{remark}

\end{document}